\documentclass{article}

\usepackage{tikz}
\usepackage{amssymb}
\usepackage{amsmath}
\usepackage{amstext}
\usepackage{amsthm}
\usepackage{latexsym}
\usepackage{verbatim}
\usepackage[colorlinks=true,citecolor=blue,linkcolor=blue,urlcolor=blue]{hyperref}

\newcommand{\inq}{\textsf{INQ}}

\newcommand{\inqp}{\textsf{INQ}^+}
\newcommand{\inqm}{\textsf{INQ}^-}

\renewcommand{\phi}{\varphi}

\newcommand{\PDL}{\operatorname{PDL}}

\newcommand{\gap}{[\phantom{a}]}

\newcommand\D{\textsf{D}}

\newcommand{\w}{\overline{w}}

\renewcommand{\L}{\mathcal{L}}

\def\disji{\rotatebox[origin=c]{-90}{$\geqslant$}}
\newcommand{\lori}{\,\disji\,}

\newcommand{\dep}[2]{=\hspace{-3pt}({#1};{#2})}
\newcommand{\con}[1]{=\hspace{-3pt}({#1})}

\theoremstyle{definition}
\newtheorem{theorem}{Theorem}[section]

\newtheorem{lemma}[theorem]{Lemma}

\title{Two strong undefinability results in inquisitive and team semantics}
\author{Fausto Barbero \\ (University of Helsinki)}

\date{}

\begin{document}
\maketitle

\begin{abstract}
We prove two (strong) undefinability results for logics based on inquisitive semantics (or its variant, team semantics). Namely: 1) we show the undefinability of intuitionistic implication in extended propositional inquisitive logic with tensor; 2) we show the undefinability of global disjunction in extended propositional dependence logic.
\end{abstract}

The purpose of the present note is to make publicly available two undefinability results that were obtained during the final stages of preparation of the paper \cite{CiaBar2019} and could not be included in the final manuscript. The results concern \emph{extended propositional inquisitive logic with tensor} (the extension of propositional inquisitive logic with the tensor disjunction $\otimes$ and with the definable operators $\top,\neg$ and $?$) and \emph{extended propositional dependence logic} (whose syntax includes atoms expressing dependencies between arbitrary classical formulas).

Propositional inquisitive semantics evaluates formulas over \emph{sets} of possible worlds and propositional team semantics does it over \emph{sets} of finite Boolean valuations.\footnote{Thus, the models considered in propositional team semantics are a special case of inquisitive models. However, for all the languages considered in this note, the two semantics frameworks produce equivalent logics.} Evaluating formulas over sets of worlds or valuations leads to a more complex theory of the (un)definability of connectives than what is attained with the usual notion of satisfaction at a single world (resp., a single assignment). It is well-known from the literature that the following two concepts do \emph{not} coincide in inquisitive/team semantics (while they do in classical logic).
\begin{itemize}
\item A binary connective $\circ$ is \emph{expressible}\footnote{In the literature on team semantics the word ``definable'' is often used with this meaning.} in a language $\L$ if, for every $\psi,\chi\in\L$, there is a $\varphi\in\L$ that is equivalent to $\psi \circ \chi$. 
\item A binary connective $\circ$ is \emph{definable}\footnote{In the literature on team semantics the word ``uniformly definable'' is often used with this meaning.} in a language $\L$ if there is a template formula $\varphi(\phantom{a},\phantom{a})$ such that, for every $\psi,\chi\in\L$, $\psi \circ \chi$ is equivalent to $\varphi(\psi,\chi)$.\footnote{The template formula $\varphi(\phantom{a},\phantom{a})$ can be seen as a formula containing some occurrences of ``gap symbols'' $\gap_L$ and $\gap_R$. In $\varphi(\psi,\chi)$, each occurrence of $\gap_L$ is replaced with $\psi$ and each occurrence of $\gap_R$ is replaced with $\chi$.}
\end{itemize} 
While (in)expressibility results typically follow as corollaries of characterizations of the expressive power of languages, an undefinability proof typically requires more sophisticate, and often \emph{ad hoc}, techniques (see for example \cite{Gal2013}, \cite{Yan2017} for special results, and \cite{Cia2016}, \cite{HelLuoSanVir2014}, \cite{Ron2018}, \cite{LucVil2019},  \cite{HelLuoVaa2024} for more systematic approaches). We can make more explicit the meaning of \emph{un}definability of a binary connective $\circ$ as follows:
\begin{quote}
For every template formula $\varphi(\phantom{a},\phantom{a})$, there are formulas $\psi_\varphi,\chi_\varphi$ such that $\varphi(\psi_\varphi,\chi_\varphi)\not\equiv  \psi_\varphi\circ\chi_\varphi$.
\end{quote} 
However, the results in \cite{CiaBar2019} and in the present note establish something more. These results have the following form:
\begin{quote}
There are $\psi,\chi$ and a model $M$ such that, for every template formula $\varphi(\phantom{a},\phantom{a})$,   $\varphi(\psi',\chi')\not\equiv_{M'} \psi'\circ\chi'$, where $\psi',\chi'$ differ from $\psi,\chi$ only by renaming of atoms and $M'$ is isomorphic to $M$.
\end{quote}
In other words, there is a single model (up to isomorphism) where all putative definitions of $\circ$ fail (and this failure is essentially witnessed by the same substitution instance in each case).\footnote{In all the results obtained so far, a small finite model sufficed.}
We will refer to a result of this kind as a proof of \emph{strong undefinability}. In this note we establish two results of this kind.

\begin{enumerate}
\item The paper \cite{CiaBar2019} proved the independence of the connectives of propositional inquisitive logic plus the tensor operator. Here we consider a further extension of the logic with the operators  $\top,\neg$ and $?$ that are commonly used in the literature on inquisitive logics. While these operators are definable in inquisitive logic (plus tensor), it is not trivial to see whether the other connectives remain independent in this context. The arguments from \cite{CiaBar2019} can be reused to prove the independence of $\land$ and $\otimes$, but the independence of $\rightarrow$ (inquisitive implication) requires a new argument (section \ref{sec: undefinability of implication}).

\item In \cite{Yan2017} it is proved, via a complex lemma, that the global disjunction $\lori$ (typical of inquisitive logic) is not definable in propositional dependence logic. The result concerns propositional dependence logic in its most common presentation, which allows for dependence atoms of the form $\dep{p_1,\dots,p_n}{q}$, which say that the truth value of a propositional letter $q$ is functionally determined by the truth values of $p_1,\dots,p_n$. The argument used in \cite{Yan2017} does not work if the logic is extended with atoms of the form $\dep{\alpha_1,\dots,\alpha_n}{\beta}$\footnote{Fan Yang, personal communication.} (dependence among truth values of classical formulas) and, as far as we know, it is still an open problem whether $\lori$ is definable in this context. We prove that it is not (section \ref{sec: undefinability of global disjunction}). Our argument also provides a new and much simpler proof for the case of propositional dependence logic.
\end{enumerate}

\noindent Sections \ref{sec: syntax} and \ref{sec: semantics} present the syntax and the semantics of the languages considered in this note.

\section{Syntax}\label{sec: syntax}

We will consider 0-ary connectives $\bot,\top$, unary connectives $\neg$ and $?$, and binary connectives $\land,\lori,\rightarrow$ and $\otimes$.

Given a list of connectives $\circ_1, \dots, \circ_n$ we will denote as $\L[\circ_1, \dots, \circ_n]$ the language having as atomic formulas a countable set of propositional letters $p,q,\dots$ and closed under applications of the connectives (according to their arities). We will consider the following inquisitive languages:

\begin{itemize}
\item $\inq$ (propositional inquisitive logic): $\L[\land,\lori,\rightarrow,\bot]$
\item $\inq^\otimes$ (propositional inquisitive logic with tensor): $\L[\land,\lori,\otimes,\rightarrow,\bot]$
\item $\inqp$ (extended $\inq^\otimes$): $\L[\land,\lori,\otimes,\to, \bot,\top,\neg,?]$
\item  $\inqm$ ($\inqp$ without $\rightarrow$): $\L[\land,\lori,\otimes,\bot,\top,\neg,?]$.
\end{itemize}

The family of logics of dependence follows more unfamiliar syntax rules, and thus we define the languages of interest in the more explicit BNF form. We consider the following:

\begin{itemize}
\item $\D$ (propositional dependence logic)
\[
\varphi::= p \ | \ \neg p \ | \ \bot \ | \ \dep{p_1,\dots,p_n}{q} \ | \ \varphi \land \varphi \ | \ \varphi \otimes \varphi
\]  
where, $p,q,p_1,\dots,p_n$ are propositional letters. Notice that $\neg$ is allowed to occur only in front of propositional letters. We will often write   $\dep{\vec p}{q}$ for the \emph{dependence atom} $\dep{p_1,\dots,p_n}{q}$.  If $\vec p$ is the empty list, we simply write $\con{q}$ and call it a \emph{constancy atom}.
  
\item $\D^+$ (extended propositional dependence logic)
\[
\varphi::= p \ | \ \neg p \ | \ \bot \ | \ \dep{\alpha_1,\dots,\alpha_n}{\beta} \ | \ \varphi \land \varphi \ | \ \varphi \otimes \varphi
\]
where $\alpha_1,\dots,\alpha_n,\beta$ are \emph{classical formulas}, i.e. $\D$ formulas without dependence atoms. The formula $\dep{\alpha_1,\dots,\alpha_n}{\beta}$ will be called a \emph{generalized dependence atom}.
\end{itemize} 

\noindent The syntax of $\D$ and $\D^+$ does not allow us to write the formula $\neg\bot$; in the context of these languages, we will write $\top$ as an abbreviation for $p\otimes\neg p$.

We will also use the following abbreviation: if $\vec\psi = (\psi_1,\dots,\psi_n)$ is a tuple of formulas, then $\bigwedge \psi$ will denote their conjunction.\footnote{The operator $\land$, when given its semantics, will be associative; it is then irrelevant to specify how, exactly, parentheses are used in $\bigwedge \psi$.}

\section{Semantics}\label{sec: semantics}

We will work with inquisitive semantics, although in principle all the results in this note could be expressed and proved using team semantics. The models of inquisitive semantics are pairs $M=(W, V)$, where $W$ is a set (of ``possible worlds'') and $V$ is a valuation function that assigns to each world $w\in W$ a set of propositional letters (the set of letters ``true at $w$''). Each formula is evaluated in a model $M=(W, V)$ at an \emph{information state} (or simply \emph{state}) $s\subseteq W$. We say that a formula $\varphi$ is supported by state $s$ in model $M$ ($M,s\models \varphi$) if $M,s$ and $\varphi$ are in the relation defined by the inductive clauses below. For the purposes of this paper, we can and will usually omit reference to the model $M$, writing just $s\models \varphi$. 

\begin{itemize}
\item $s\models p\iff p \in V(w)$ for all $w\in s$
\item $s\models \bot\iff s=\emptyset$
\item $s\models \psi\land\chi\iff s\models\psi$ and $s\models\chi$
\item $s\models \psi\lori\chi\iff s\models\psi$ or $s\models\chi$
\item $s\models \psi\to\chi\iff \forall t\subseteq s: t\models\psi$ implies $t\models\chi$
\item $s\models\psi\otimes\chi\iff\exists t_1,t_2\text{ s.t.\ } t_1\models\psi, t_2\models\chi\text{ and }s=t_1\cup t_2$.
\item $s\models \top$ unconditionally
\item $s\models \neg \psi \iff s\models \psi \to \bot  \iff t\not\models \psi$ for all $t\subseteq s$.
\item $s\models ?\psi \iff s\models \psi \lori \neg\psi  \iff s\models \psi$ or $s\models\neg\psi$.

\item $s\models \dep{\vec p}{q} \iff$ for all $w,w'\in s$, $w(\vec p)=w'(\vec p)$ implies $w(q)=w'(q)$. 
\item $s\models \dep{\vec \alpha}{\beta}$ iff for all $w,w'\in s$, if $\{w\}\models \bigwedge\vec\alpha \iff \{w'\}\models \bigwedge\vec\alpha$, then $\{w\}\models \beta\iff\{w'\}\models \beta$.
\end{itemize}
Notice in particular that, for a propositional letter $p$, $s\models \neg p$ iff $p\notin V(w)$ for each $w\in s$; and that a constancy atom $\con \beta$ is just an alternative notation for $?\beta$.

We say that two formulas $\psi,\chi$ are \emph{equivalent in $M=(W,V)$} (and we write $\psi\equiv_M\chi$) if $M,s\models \psi$ iff $M,s\models \chi$ for all $s\subseteq W$. We say that  $\psi,\chi$ are \emph{equivalent} ($\psi\equiv\chi$) iff $\psi\equiv_M\chi$ for all models $M$. 

All the languages considered in this note have the property of \emph{downwards closure}, that is: for any formula $\varphi$, if $s\models \varphi$ and $t\subseteq s$, then $t\models \varphi$. This can be shown by straightforward inductive arguments. Furthermore, the languages have the \emph{empty set property}: the empty state $\emptyset$ supports all formulas.

For any formula $\varphi$ and model $M=(W,V)$, we write $[\varphi]_M := \{s\subseteq W \mid s\models \varphi \}$; this is the \emph{inquisitive proposition} of $\varphi$. Notice that $\psi\equiv_M\chi$ iff $[\psi]_M=[\chi]_M$.
We will omit the subscript $M$ if the model is clear from the context or irrelevant.   We remark that:

\begin{itemize}
\item $[\psi \land \chi] = [\psi]\cap[\chi]$.
\item $[\psi \lor \chi] = [\psi]\cup[\chi]$.
\item $[\psi \otimes \chi] = \{s \cup t \mid s,t \subseteq W \text{ and } s\in [\psi], t\in [\chi]\}$.
\end{itemize}

\noindent We will also write $[\psi]\otimes[\chi]$ for this latter set.

\section{Independence of inquisitive implication in $\inqp$}\label{sec: undefinability of implication}

In \cite{CiaBar2019} it was shown that in $\inq^\otimes$  the connectives $\land$ and $\otimes$ are independent, or more precisely strongly undefinable from the other connectives in the language. The exact same arguments show that these connectives are independent in $\inqp$.

On the contrary, the argument given in \cite{CiaBar2019} to show the independence of $\to$ in $\inq^\otimes$ used a specific characteristic of language $\inq^\otimes$: the fact that all of its connectives, except for $\rightarrow$, are increasing monotone in both arguments. This kind of argument is not suitable anymore for $\inqp$, where also the operators $\neg$ and $?$ lack this property. We provide instead a direct argument, showing that $\rightarrow$ is strongly undefinable in $\inqm$.




\vspace{10pt}

For each pair of (distinct) atoms $p,q$, we define a corresponding possible world model $M_{pq} = (W,V_{pq})$, where:
\begin{itemize}
\item $W = \{w_1,w_2,w_3\}$
\item $V_{pq}(w_1) = \{p\}$ 
\item $V_{pq}(w_2) = \{q\}$
\item $V_{pq}(w_3) = \emptyset$.  
\end{itemize}
We will show that no $\inqm$ formula of the form $\varphi(a,b)$ without occurrences of $p,q$ is such that $\varphi(?p,?q)\equiv_{M_{pq}} ?p\to ?q$.

\begin{lemma}\label{LEMWI}
Let $\varphi(a,b)$ be an $\inqm$ formula without occurrences of $p,q$. If $\varphi(?p,?q)\not\equiv_{M_{pq}}\bot$,  then $\{w_i\}\models\varphi(?p,?q)$ for each $i=1,2,3$.
\end{lemma}

\begin{proof}
 We first notice that, without loss of generality, we can assume the symbol $?$ not to occur in $\varphi(a,b)$ (just replace each subformula of the form $?\theta$ with $\theta\lori\neg\theta$).
We then proceed by induction on the syntax of $\varphi(a,b)$. Write $\varphi^*$ for $\varphi(?p,?q)$ and similarly for other formulas.

Base cases: the cases for $\varphi(a,b)$ being $\bot$ or $\top$ are straightforward. In case  $\varphi(a,b)$ is $a$, then $\varphi^*$ is $?p$, and each of the $\{w_i\}$ support $?p$. The case for $\varphi(a,b)$ being $b$ is analogous. Finally, suppose $\varphi(a,b)$ is an atom $r$ distinct from $a,b$ (and, by assumption, distinct from $p$ and $q$). Since $r$ is not true at any world in $M_{pq}$, by downward closure we have $\varphi^*\equiv_{M_{pq}}\bot$.

Inductive step: we have several cases.

\begin{itemize}
\item $\varphi(a,b)$ is of the form $\theta(a,b)\land\eta(a,b)$. If either $\theta^*$ or $\eta^*$ is $\equiv_{M_{pq}}\bot$, then also $\varphi^*$ is. Otherwise, by the inductive hypothesis $\{w_i\}\in [\theta^*]\cap[\eta^*]=[\theta^* \land\eta^*] = [\varphi^*]$ for each $i=1,2,3$. 

\item $\varphi(a,b)$ is of the form $\theta(a,b)\lori\eta(a,b)$. If $\theta^*\equiv_{M_{pq}}\bot$, then $\varphi^*\equiv_{M_{pq}}\eta^*$, so $\varphi^*$ satisfies the statement by inductive hypothesis. 
Otherwise, by the inductive assumption $\{w_i\}\in [\theta^*]$; therefore,  $\{w_i\}\in [\theta^*]\cup [\eta^*] = [\theta^*\lori \eta^*]= [\varphi^*]$.
 
\item $\varphi(a,b)$ is of the form $\theta(a,b)\otimes\eta(a,b)$. 
If both $\theta^*$ and $\eta^*$ are $\equiv_{M_{pq}}\bot$, then  $\varphi^*$ is. Otherwise, wlog assume that $\theta^*\not\equiv_{M_{pq}}\bot$. Then, by the inductive assumption on $\theta^*$, we have $\{w_i\}\models \theta^*\models \theta^*\otimes\eta^*$. 

\item $\varphi(a,b)$ is of the form $\neg\theta(a,b)$. If  $\theta^*\equiv_{M_{pq}}\bot$, then $\varphi^*\equiv \top$, so the  $\{w_i\}$ satisfy this formula. Otherwise, by the inductive assumption each $\{w_i\}$ satisfies $\theta^*$; but then by downward closure $\varphi^*\equiv_{M_{pq}}\bot$.  
\end{itemize}
\end{proof}

\noindent Given any $S\in \wp(\wp(W))$, we write $S^{\downarrow} := \{t\subseteq s  \mid s\in S \}$ (\emph{downward closure} of $S$). It can be verified that in $M_{pq}$ we have:

\begin{itemize}
\item $[?p\lori ?q] = \{\{w_1,w_3\},\{w_2,w_3\}\}^\downarrow$.
\item  $A:= [?p\rightarrow ?q] = \{\{w_1,w_2\},\{w_1,w_3\}\}^\downarrow$.
\item $B:= [?q\rightarrow ?p] = \{\{w_1,w_2\},\{w_2,w_3\}\}^\downarrow$.
\item $C:= [?p\rightarrow ?q]\cap [?q\rightarrow ?p] = \{\{w_1,w_2\}, \{w_3\}\}^\downarrow$.  
\end{itemize}

\noindent We want to show that no $\inqm$ formula of the form $\varphi(?p,?q)$ can be equivalent to $?p \to ?q$ in $M_{pq}$ (in case $\varphi(a,b)$ has no occurrences of $p,q$). In order to have a stronger inductive assumption, we prove more generally that the inquisitive proposition of $\varphi(?p,?q)$ cannot be $A,B$ or $C$. 

\begin{lemma}\label{LEMABC}
Let $\varphi(a,b)$ be an $\inqm$ formula without occurrences of $p,q$. Then $[\varphi(?p,?q)] \neq A,B,C$.
\end{lemma}

\begin{proof}
By induction on the syntax of $\varphi(a,b)$. As before, we can assume that this formula contains no occurrences of $?$, and we write $\varphi^*$ for $\varphi(?p,?q)$. 

The base cases are easy.

Inductive step: we have several cases.

\begin{itemize}
\item $\varphi(a,b)$ is of the form $\theta(a,b)\land\eta(a,b)$. 
\begin{itemize}
\item Observe that there is only one proper subfamily of $\wp(W)$ which extends $A$ (namely, $\wp(W)\setminus\{W\} = \{\{w_1,w_2\},\{w_2,w_3\},\{\w_1,w_3\}^\downarrow$). Therefore, $A$ cannot be obtained as intersection of two subsets of $\wp(W)$ which are both distinct from $A$. Since, by the  inductive hypothesis, we also have that $[\theta^*]\neq A \neq [\eta^*]$, we can then conclude that $[\varphi^*] = [\theta^*]\cap  [\eta^*] \neq A$. By a similar argument we conclude $[\varphi^*] \neq B$.
\item The fact that $[\varphi^*] \neq C$ is proved similarly, observing first that the only way of writing $C$ as an intersection of subsets of $\wp(W)$ which strictly include $C$ is $A\cap B$.
\end{itemize}

\item $\varphi(a,b)$ is of the form $\theta(a,b)\lori\eta(a,b)$.
\begin{itemize}
\item If $\theta^*\equiv_{M_{pq}}\bot$, then $\varphi^*\equiv_{M_{pq}}\eta^*$; and $[\eta^*]\neq A,B,C$ by the inductive hypothesis. In case instead $\eta^*\equiv_{M_{pq}}\bot$, we have a similar proof.
\item Suppose neither $\theta^*$ nor $\eta^*$ is $\equiv_{M_{pq}}\bot$. We need to check a few subcases.
\begin{itemize}
\item Suppose for the sake of contradiction that $[\varphi^*] =C$. Then, either $[\theta^*]$ or $[\eta^*]$ contains $\{w_1,w_2\}$; say, $\{w_1,w_2\}\in [\theta^*]$. By lemma \ref{LEMWI} we have $\{w_3\}\in [\theta^*]$. So, $C\subseteq [\theta^*]$. Since furthermore $[\theta^*]\subseteq [\varphi^*]=C$, we must conclude that $[\theta^*]=C$, which contradicts the inductive hypothesis.
\item Suppose for the sake of contradiction that $[\varphi^*] =A$. As before, we have, say, that  $\{w_1,w_2\}\in [\theta^*]$. By lemma \ref{LEMWI} we have $\{w_3\}\in [\theta^*]$. Since furthermore $[\theta^*]\subseteq [\varphi^*]=A$,  we must conclude that $[\theta^*]$ is either $C$ or $A$, which contradicts the inductive hypothesis.
\item The case for $[\varphi^*] =B$ is similar.
\end{itemize}
\end{itemize}

\item $\varphi(a,b)$ is of the form $\theta(a,b)\otimes\eta(a,b)$. The subcase in which either $\theta^*\equiv_{M_{pq}}\bot$ or $\eta^*\equiv_{M_{pq}}\bot$ is treated as in the $\lori$ case. If instead neither of these holds, by lemma \ref{LEMWI} both $\theta^*$ and $\eta^*$ are satisfied by each of the $\{w_i\}$. But then $\{w_i,w_j\}\in [\theta^*\otimes\eta^*]$ for each $i,j = 1,2,3$ with $i\neq j$; so $[\theta^*\otimes\eta^*]\neq A,B,C$.

\item $\varphi(a,b)$ is of the form $\neg\theta(a,b)$.
If $\theta^* \equiv_{M_{pq}}\bot$, then $\varphi^*\equiv_{M_{pq}}\top$; and $[\top]\neq A,B,C$. Otherwise, by lemma \ref{LEMWI}, $\{w_i\}\models \theta^*$ for each $i=1,2,3$. But then, by downward closure,  $[\varphi^*]= \{\emptyset\}\neq A,B,C$.
\end{itemize}
\end{proof}

\begin{theorem}\label{thm: undefinability of implication}
The connective $\to$ is not definable in $\inqm$.
\end{theorem}

\begin{proof}
Suppose for the sake of contradiction that there is a formula $\varphi(a,b)\in\inqm$ such that, for all formulas $\eta,\theta\in\inqm$, $\varphi(\eta,\theta)\equiv \eta\rightarrow\theta$. Let $p,q$ be two distinct atoms which do not occur in $\varphi(a,b)$. Then, as a special case of lemma \ref{LEMABC}, we obtain that $[\varphi(?p,?q)]_{M_{pq}}\neq [?p\to?q]_{M_{pq}}$, i.e. $\varphi(?p,?q) \not\equiv_{M_{pq}} \hspace{2pt} ?p\to?q$. Thus $\varphi(?p,?q) \not\equiv \hspace{2pt} ?p\to?q$, contradicting our initial assumption. 
\end{proof}




\section{Strong undefinability of $\lori$ in (extended) propositional dependence logic}\label{sec: undefinability of global disjunction}


In this section we will prove that the inquisitive disjunction\footnote{Also variously known as \emph{global}, \emph{Boolean} or \emph{intuitionistic} disjunction.} $\lori$ is strongly undefinable in propositional dependence logic $\D$, and later show how this proof can be adapted to extended propositional dependence logic $\D^+$. Undefinability simpliciter was shown for $\D$ in \cite{Yan2017}  by a more informative but quite complicated method, which does not straightforwardly extend to $\D^+$.

For any distinct propositional letters $p,q$ we consider a model $M_{pq} = (W, V_{pq})$ where:
\begin{itemize}
\item $W=\{w_1,w_2,w_3\}$
\item $[p] = \{w_1,w_2\}^\downarrow$
\item $[q] = \{w_2,w_3\}^\downarrow$
\item $[r]=\{\emptyset\}$ for each $r\neq p,q$
\end{itemize}

Differently from the case of inquisitive logic, not every $\D$ formula can be accepted as a definition of a connective; we must ensure that substitution instances are still well-formed $\D$ formulas.
Towards this goal, we call $\varphi(a,b)$ a \emph{context} if it is a $\D$ formula in which $a$ and $b$ do not occur negated nor in a dependence atom.\footnote{The terminology is from \cite{Yan2017}.}

We want to show that, for every context $\varphi(a,b)$ without occurrences of $p,q$, $\varphi(\con p,\con q)\not\equiv_{M_{pq}} \con p\lori \con q$.

\begin{lemma}\label{DEPLEMMA}
Let $\delta$ be a dependence atom without occurrences of $p,q$. Then $\delta\equiv_{M_{pq}}\top$.
\end{lemma}

\begin{lemma}\label{LEMMACHAR}
Let $\varphi(a,b)$ be a context without occurrences of $p,q$. Then the formula $\varphi(\con p, \con q)$ is equivalent in $M_{pq}$ to one of the following:
\begin{itemize}
\item $\top$
\item  $(\con p\land \con q)\otimes (\con p\land \con q)$
\item $\con p$
\item $\con q$
\item $\con p\land \con q$
\item $\bot$
\end{itemize}
\end{lemma}

\begin{proof}
By induction on $\varphi$. Write $\varphi^*$ for $\varphi(\con p, \con q)$, and similarly for other formulas.
\begin{itemize}
\item Base cases.
\begin{itemize}
\item  If $\varphi(a,b)$ is $a$, then $\varphi^*$ is $\con p$, and we are done. Similarly if  $\varphi(a,b)$ is $b$.
\item  If $\varphi(a,b)$ is $r$ for some $r\neq a,b,p,q$, then $\varphi^*$ is $r$, so  $\varphi^*\equiv_{M_{pq}} \bot$.
\item  If $\varphi(a,b)$ is $\neg r$ for some $r\neq a,b,p,q$, then $\varphi^*$ is $\neg r$, so  $\varphi^*\equiv_{M_{pq}} \top$.
\item  If $\varphi(a,b)$ is a dependence atom, then $a$ and $b$ do not occur in $\varphi(a,b)$, and      $\varphi^* = \varphi(a,b)$. Thus, by lemma \ref{DEPLEMMA}, $\varphi^*\equiv_{M_{pq}}\top$.
\end{itemize}
\item Case $\varphi(a,b) = \eta(a,b) \land \theta(a,b)$. By induction hypothesis, $\eta^*$ and $\theta^*$ are equivalent in $M_{pq}$ to one out of $\top,\con p,\con q,\con p\land \con q,(\con p\land \con q)\otimes (\con p\land \con q),\bot$. We verify that the set of the inquisitive propositions of these six formulas is closed under intersection. 


Clearly no set is altered by intersection with $[\top]$, and intersecting anything with $[\bot]$ gives $[\bot]$. For the rest, observe that
\begin{itemize}
\item $[\con p]\cap[\con q] = [\con p]\cap[\con p\land \con q] = [\con q]\cap[\con p \land \con q] = [\con p \land \con q]$, where the first two equalities follow from the associativity, commutativity and idempotence of $\land$.
\item $[\con p], [\con q], [\con p\land\con q] \subset [(\con p\land \con q)\otimes (\con p\land \con q)]$, so if $\eta^*$ is either $\con p,\con q$ or $\con p\land \con q$, then $[\eta^*]\cap [(\con p\land \con q)\otimes (\con p\land \con q)]= [\eta^*]$.
\end{itemize}

\item Case $\varphi(a,b) = \eta(a,b) \otimes \theta(a,b)$. By induction hypothesis, each of $\eta^*$ and $\theta^*$ is equivalent in $M_{pq}$ to one out of $\top,\con p,\con q,\con p\land \con q,(\con p\land \con q)\otimes (\con p\land \con q),\bot$. We verify that the set of the inquisitive propositions of these six formulas is closed under the set-theoretical operator $\otimes$. Again, this is trivial if $\top$ or $\bot$ are involved.
\begin{itemize}
\item $[\con p]\otimes[\con p] = \wp(s) = [\top]$, because, first, $W= \{w_1,w_2\}\cup\{w_3\}$; secondly, $\{w_1,w_2\},\{w_3\}\in [\con p]$, so that $W\in [\con p \otimes \con p]$; and lastly, $\otimes$ preserves downward closure. Similarly for $\con q$.
\item $[\con p]\otimes[\con q] = \wp(s) = [\top]$, because $W= \{w_1,w_2\}\cup\{w_2,w_3\}$ and $\{w_1,w_2\}\in [\con p],\{w_2,w_3\}\in [\con q]$ (and using again downward closure).
\item Let $\eta^*$ be either $\con p,\con q$ or $(\con p\land \con q)\otimes (\con p\land \con q)$. Then either $\{w_1,w_2\}$ or $\{w_2,w_3\}$ is in $[\eta^*]$. Now, since both $\{w_3\}$ and $\{w_1\}$ are in $[\con p \land \con q]$, and $W=\{w_1,w_2\}\cup \{w_3\} = \{w_1\} \cup\{w_2,w_3\}$, in each of these cases $[\eta^*]\otimes[\con p \land \con q] = \wp(s) = [\top]$. For similar reasons, $[\eta^*]\otimes[(\con p \land \con q)\otimes(\con p \land \con q)] = [\top]$.
\end{itemize}
\end{itemize}
 
\end{proof}

\begin{theorem}\label{TEOINQDISJ}
The connective $\lori$ is strongly undefinable in $\D$.
\end{theorem}

\begin{proof}
Let $\varphi(a,b)$ be a context.  Let $p,q$ be two atoms which do not occur in $\varphi(a,b)$. We show that $\varphi(\con p,\con q)\not\equiv \con p \lori\con q$. Consider the model $M_{pq}$ corresponding to $p,q$. In it, $[\con p \lori\con q] = \{\{w_1,w_2\},\{w_2,w_3\}\}^\downarrow$. By lemma \ref{LEMMACHAR}, $[\varphi(\con p,\con q)]$ is either $[\top],[\con p],[\con q],[\con p\land \con q],[(\con p\land \con q)\otimes (\con p\land \con q)],[\bot]$. It is easily checked that none of these sets is $[\con p \lori\con q]$; so, $\varphi(\con p,\con q) \not\equiv_{M_{pq}} \con p \lori\con q$; and thus $\varphi(\con p,\con q) \not\equiv \con p \lori\con q$.   
\end{proof}


We briefly show that the result above easily extends to a proof of strong undefinability of $\lori$ in $\D^+$, the language that differs from $\D$ only in that it allows dependence atoms of the form $\dep{\vec \alpha}{\beta}$, where $\vec \alpha,\beta$ are $\D$ formulas without occurrences of dependence atoms.

First of all, the analogue of lemma \ref{DEPLEMMA} 
can be proved by pointing out that any (generalized) dependence atom without occurrences of $p,q$ is of the form $\dep{\vec\alpha}{\beta}$, where $\beta$ is a formula without occurrences of $p,q$. Since in $M_{pq}$ all worlds agree on the truth value of propositional letters different from $p,q$, it can be shown by a straightforward induction that all worlds in $M_{pq}$ agree about the truth value of $\beta$. Therefore, $\dep{\vec\alpha}{\beta}$ is trivially satisfied by all substates of $W$, i.e.  $\dep{\vec\alpha}{\beta}\equiv_{M_{pq}} \top$.

Secondly, observe that, in the analogue of lemma \ref{LEMMACHAR}, the generalized dependence atoms play a role only in the base case for dependence atoms. But this case is taken care by the analogue of lemma \ref{DEPLEMMA} exactly in the same way as in the original proof. 

Thirdly, the generalized dependence atoms play no role in the rest of the proof of theorem \ref{TEOINQDISJ}. Thus we have:

\begin{theorem}\label{TEOINQDISJPLUS}
The connective $\lori$ is strongly undefinable in $\D^+$.
\end{theorem}


\bibliography{iilogics2}
\bibliographystyle{plain}

\end{document}